\documentclass{amsart}
\newtheorem{thm}{Theorem}[section]
\newtheorem{prop}[thm]{Proposition}
\newtheorem{lem}[thm]{Lemma}
\theoremstyle{definition}
\newtheorem{ex}[thm]{Example}
\theoremstyle{remark}
\newtheorem{rem}[thm]{Remark}
\newcommand{\N}{\mathbb{N}}
\newcommand{\Q}{\mathbb{Q}}
\newcommand{\Z}{\mathbb{Z}}
\DeclareMathOperator{\Ima}{Im}
\DeclareMathOperator{\ord}{ord}
\begin{document}
\title[Characterization of order structures avoiding 3-APs]{Characterization of order structures avoiding three-term arithmetic progressions}
\author{Minoru Hirose}
\address{Institute for Advanced Research, Nagoya University, Furo-cho, Chikusa-ku, Nagoya, 464-8602, Japan}
\email{minoru.hirose@math.nagoya-u.ac.jp}
\author{Shingo Saito}
\address{Faculty of Arts and Science, Kyushu University, 744, Motooka, Nishi-ku, Fukuoka, 819-0395, Japan}
\email{ssaito@artsci.kyushu-u.ac.jp}
\begin{abstract}
 It is known that the set of all nonnegative integers may be equipped with a total order
 that is chaotic in the sense that there is no monotone three-term arithmetic progressions.
 Such chaotic order must be so complicated that
 the resulting ordered set cannot be order isomorphic to the set of all nonnegative integers or the set of all integers with the standard order.
 In this paper, we completely characterize order structures of chaotic orders on the set of all nonnegative integers,
 as well as on the set of all integers and on the set of all rational numbers.
\end{abstract}
\maketitle

\section{Introduction}
It has long been known (\cite{EJ,L,O,T}) that it is possible to arrange the integers $1,\dots,n$ (and consequently any finitely many rational numbers)
into a sequence $a_1,\dots,a_n$ without a three-term subsequence forming an arithmetic progression,
or more precisely without $i$, $j$, $k$ with $1\le i<j<k\le n$ and $a_j-a_i=a_k-a_j$.
In contrast, Davis, Entringer, Graham, and Simmons~\cite{DEGS} proved that
it is impossible to arrange all of the positive integers (or equivalently all of the nonnegative integers)
into a sequence $a_1,a_2,\dots$ or a two-sided sequence $\dots,a_{-1},a_0,a_1,\dots$
without a three-term subsequence forming an arithmetic progression.
On the other hand, Ardal, Brown, and Jungi\'{c}~\cite{ABJ} proved that
there exists a total order $\preceq$ on $\Q$ such that there are no $a,b,c\in\Q$ with $a\prec b\prec c$ and $b-a=c-b$,
which implies that such total order exists on $\N$ (by which we mean the set of all nonnegative integers) and on $\Z$;
it should be noted that although they invoked K\H{o}nig's infinity lemma to obtain the total order on $\Q$ (and hence the resulting total order is rather abstract),
they also gave an explicit total order on $\Z$ (and hence on $\N$).

To summarize these results, we find it convenient to introduce the following terminology.
For a set $S$ of rational numbers and a totally ordered set $(X,\preceq)$,
we say that a bijection $f\colon S\to X$ is \emph{chaotic} if there are no $a,b,c\in S$ with $f(a)\prec f(b)\prec f(c)$ and $b-a=c-b$.
Then the results of Davis, Entringer, Graham, and Simmons say that
if $S=\N$
and $X=\N,\Z$ with the standard order,
then no bijection $f\colon S\to X$ is chaotic;
the results of Ardal, Brown, and Jungi\'{c} say that
if $S=\N,\Z,\Q$, then there exist a countably infinite totally ordered set $(X,\preceq)$ and a chaotic bijection $f\colon S\to X$.

This observation leads us to the following natural question:
given $S=\N,\Z,\Q$, which countably infinite totally ordered sets $(X,\preceq)$ admit a chaotic bijection $f\colon S\to X$?
Our main theorem completely answers this question.
To state our main theorem, recall that a point $p$ in a totally ordered set $(X,\preceq)$ is said to be \emph{isolated}
if it is isolated in the topological space $X$ endowed with the order topology,
namely (assuming that $X$ has at least two elements) if either
(i) $\{x\in X\mid x\prec x_0\}=\{p\}$ for some $x_0\in X$,
(ii) $\{x\in X\mid x\succ x_0\}=\{p\}$ for some $x_0\in X$, or
(iii) $\{x\in X\mid x_0\prec x\prec x_1\}=\{p\}$ for some $x_0,x_1\in X$.

\begin{thm}[Main Theorem]\label{thm:main}
 Let $(X,\preceq)$ be a countably infinite totally ordered set.
 \begin{enumerate}
  \item There exists a chaotic bijection $f\colon\N\to X$ if and only if $X$ has no isolated points.
  \item There exists a chaotic bijection $f\colon\Z\to X$ if and only if $X$ has no isolated points.
  \item There exists a chaotic bijection $f\colon\Q\to X$ if and only if $X$ has no isolated points and either $X$ does not have a maximum or $X$ does not have a minimum.
 \end{enumerate}
\end{thm}

\section{Proof of the `only if' parts of our main theorem}
This section proves the `only if' parts of our main theorem.
Let $S\subset\Q$, and let $(X,\preceq)$ be a countably infinite totally ordered set.
The key to the proof is the observation that
for $S\in\{\N,\Z,\Q\}$, every chaotic map $f\colon S\to X$ has a stronger property concerning the $2$-adic order.

\subsection{Chaotic maps and binary maps}
We begin by generalizing the definition of chaotic maps:
we say that $f\colon S\to X$ is \emph{chaotic} if $f$ is injective and there are no $a,b,c\in S$ such that $f(a)\prec f(b)\prec f(c)$ and $b-a=c-b$.

We then define a property that is stronger than being chaotic.
Recall that the \emph{$2$-adic order} $\ord_2r$ of $r\in\Q^{\times}=\Q\setminus\{0\}$ is defined as the unique $n\in\Z$
such that $2^{-n}r$ can be written as a quotient of odd integers;
$\ord_20$ is defined as $\infty$.
We say that $f\colon S\to X$ is \emph{binary} if $f$ is injective and there are no $a,b,c\in S$ such that $f(a)\prec f(b)\prec f(c)$ and $\ord_2(b-a)=\ord_2(c-b)$.
It is obvious that every binary map $f\colon S\to X$ is chaotic.

\begin{ex}\label{ex:chaotic}
 If $f\colon\{0,1,2,3,4,5,6,7\}\to X$ satisfies
 \[
  f(0)\prec f(4)\prec f(2)\prec f(6)\prec f(1)\prec f(5)\prec f(3)\prec f(7),
 \]
 then $f$ is binary.
 If $f\colon\{0,1,2,3\}\to X$ satisfies
 \[
  f(2)\prec f(3)\prec f(0)\prec f(1),
 \]
 then $f$ is chaotic but not binary.
\end{ex}

\begin{prop}\label{prop:binary_implies_no_isolated_points}
 If $S\in\{\N,\Z,\Q\}$ and there exists a binary bijection $f\colon S\to X$, then $X$ has no isolated points.
\end{prop}

\begin{proof}
 Assume that $\{x\in X\mid f(a)\prec x\prec f(b)\}=\{f(c)\}$ for some $a,b,c\in S$.
 Take $c'\in S\setminus\{c\}$ with $\ord_2(c'-c)>\ord_2(a-c)$ and $\ord_2(c'-c)>\ord_2(b-c)$
 (we can for example take $c'=c+2^n$ for a sufficiently large $n\in\N$).
 Since $f(c)$ is the only element of $X$ between $f(a)$ and $f(b)$,
 it must be the case that either $f(c')\prec f(a)\prec f(c)$ or $f(c)\prec f(b)\prec f(c')$.
 Since $\ord_2(a-c')=\ord_2(c-a)$ and $\ord_2(b-c)=\ord_2(c'-b)$,
 the map $f$ cannot be binary, which is a contradiction.

 In a similar manner, both assuming that $\{x\in X\mid x\succ f(a)\}=\{f(c)\}$ for some $a,c\in S$ and assuming that $\{x\in X\mid x\prec f(a)\}=\{f(c)\}$ for some $a,c\in S$
 lead to a contradiction.
\end{proof}

\begin{rem}
 Proposition~\ref{prop:binary_implies_no_isolated_points} remains valid for any set $S\subset\Q$ with the property that
 for every $a\in S$, the set $\{\ord_2(b-a)\mid b\in S\setminus\{a\}\}$ is unbounded from above.
\end{rem}

\subsection{Chaotic maps are binary if $S\in\{\N,\Z,\Q\}$}
In this subsection, we prove that if $S\in\{\N,\Z,\Q\}$, then every chaotic map $f\colon S\to X$ is in fact binary (Proposition~\ref{prop:chaotic_implies_binary}).
Note that this is not the case for general $S$, as the second example in Example~\ref{ex:chaotic} shows.
We write $\N_+$ for the set of all positive integers.

\begin{lem}\label{lem:1_implies_t}
 Let $f\colon S\to X$ be chaotic and $t\in\N_+$ be odd.
 Suppose that $a\in\Q$ and $d\in\Q^{\times}$ are such that $a+id\in S$ for all $i\in\N$.
 Then $f(a)\prec f(a+d)$ if and only if $f(a)\prec f(a+td)$.
\end{lem}

\begin{proof}
 The proof proceeds by induction on $t$.
 The lemma is trivial for $t=1$.
 Suppose that the lemma is true for an odd positive integer $t$.
 In order to prove that the lemma is true for $t+2$, we only have to prove that $f(a)\prec f(a+d)$ implies $f(a)\prec f(a+(t+2)d)$.

 Set $b=a+2(t+2)d\in S$.
 Since $f$ is chaotic and $f(a)\prec f(a+d)$, we have
 \[
  f(a)\prec f(a+d)\succ f(a+2d)\prec\dots\succ f(b)\prec f(b+d).
 \]
 It follows from the inductive hypothesis applied to $b$ and $d$ that $f(b)\prec f(b+td)$
 (note that $b+id=a+(2t+4+i)d\in S$ for all $i\in\N$).

 If $f(b+td)\prec f(b+(t+2)d)$, then we have $f(b)\prec f(b+td)\prec f(b+(t+2)d)$, and so
 \[
  f(b+(t+2)d)\succ f(b)\prec f(b-(t+2)d)=f(a+(t+2)d)\succ f(a),
 \]
 as required.

 If $f(b+td)\succ f(b+(t+2)d)$, then we have
 \begin{align*}
  f(b+(t+2)d)&\prec f(b+td)\succ f(b+(t-2)d)\prec\dots\\
  &\prec f(b-(t+2)d)=f(a+(t+2)d)\succ f(a+td).
 \end{align*}
 Since $f(a)\prec f(a+td)$, we may conclude that $f(a)\prec f(a+(t+2)d)$.
\end{proof}

\begin{lem}\label{lem:1_implies_st}
 Let $f\colon S\to X$ be chaotic, $s\in\N$ be even, and $t\in\N_+$ be odd.
 Suppose that $a\in\Q$ and $d\in\Q^{\times}$ are such that $a+id\in S$ for all $i\in\N$.
 Then $f(a)\prec f(a+d)$ if and only if $f(a+sd)\prec f(a+td)$.
\end{lem}

\begin{proof}
 It suffices to prove that $f(a)\prec f(a+d)$ implies $f(a+sd)\prec f(a+td)$.
 Since $f$ is chaotic, we have
 \[
  f(a)\prec f(a+d)\succ f(a+2d)\prec f(a+3d)\succ\dotsb.
 \]
 If $s<t$, then Lemma~\ref{lem:1_implies_t} and the fact that $f(a+sd)\prec f((a+sd)+d)$ imply that $f(a+sd)\prec f((a+sd)+(t-s)d)=f(a+td)$.
 If $t<s$, then Lemma~\ref{lem:1_implies_t} and the fact that $f(a+td)\succ f((a+td)+d)$ imply that $f(a+td)\succ f((a+td)+(s-t)d)=f(a+sd)$.
\end{proof}

\begin{prop}\label{prop:chaotic_implies_binary}
 If $S\in\{\N,\Z,\Q\}$ and $f\colon S\to X$ is chaotic, then $f$ is binary.
\end{prop}

\begin{proof}
 Suppose that $f$ is not binary.
 Then there exist $a,b,c\in S$ such that $f(a)\prec f(b)\prec f(c)$ and $\ord_2(b-a)=\ord_2(c-b)=n$, say.
 Take the smallest $m\in\N_+$ such that $2^{-n}m(b-a)$ and $2^{-n}m(c-b)$ are both (necessarily odd) integers.
 Put $d=2^n/m$.

 If $b=\min\{a,b,c\}$, then applying Lemma~\ref{lem:1_implies_st} to $b$ and $d$ shows that
 $f(b)\succ f(a)=f(b+2^{-n}m(a-b)d)$ implies $f(b)\succ f(b+d)$,
 whereas $f(b)\prec f(c)=f(b+2^{-n}m(c-b)d)$ implies $f(b)\prec f(b+d)$;
 this is a contradiction.

 If $a=\min\{a,b,c\}$, then applying Lemma~\ref{lem:1_implies_st} to $a$ and $d$ shows that
 $f(a)\prec f(b)=f(a+2^{-n}m(b-a)d)$ implies $f(a)\prec f(a+d)$,
 whereas $f(b)=f(a+2^{-n}m(b-a)d)\prec f(c)=f(a+2^{-n}m(c-a)d)$ implies $f(a+d)\prec f(a)$;
 this is a contradiction.

 If $c=\min\{a,b,c\}$, then a similar argument leads us to a contradiction.
\end{proof}

Propositions~\ref{prop:binary_implies_no_isolated_points} and \ref{prop:chaotic_implies_binary} show that
if $S\in\{\N,\Z,\Q\}$ and there exists a chaotic bijection $f\colon S\to X$, then $X$ has no isolated points.
We may now complete the proof of the `only if' parts of our main theorem by noticing the following simple proposition:

\begin{prop}
 If $f\colon\Q\to X$ is a chaotic bijection, then either $X$ does not have a maximum or $X$ does not have a minimum.
\end{prop}

\begin{proof}
 If $X$ has both a maximum and a minimum, say $f(a)$ and $f(b)$, then setting $c=(a+b)/2$,
 we have $f(b)\prec f(c)\prec f(a)$ and $c-b=a-c$, a contradiction.
\end{proof}

\section{Proof of the `if' parts of our main theorem}
This section proves the `if' parts of our main theorem.
For $S\subset\Q$ and $r\in\Q$, we write $S+r=\{a+r\mid a\in S\}$.

\subsection{Construction of binary maps from $\N$ and $\Z$}
\begin{lem}\label{lem:S_and_S+r}
 Let $S\subset\Q$ and $r\in\Q^{\times}$.
 \begin{enumerate}
  \item If $\ord_2(a-b)\ne\ord_2r$ for all $a,b\in S$, then $S\cap(S+r)=\emptyset$.
  \item Suppose that $\ord_2(a-b)<\ord_2r$ for all distinct $a,b\in S$.
   Then for distinct $a,b\in S\cup(S+r)$,
   we have $\ord_2(a-b)\le\ord_2r$ with equality if and only if $a-b=\pm r$.
  \item Suppose that $\ord_2(a-b)>\ord_2r$ for all $a,b\in S$.
   Then for distinct $a,b\in S\cup(S+r)$,
   we have $\ord_2(a-b)\ge\ord_2r$ with equality if and only if either $a\in S$ and $b\in S+r$ or $a\in S+r$ and $b\in S$.
 \end{enumerate}
\end{lem}

\begin{proof}
 \begin{enumerate}
  \item 
   Obvious.
  \item
   Let $a,b\in S\cup(S+r)$ be distinct.
   If either $a,b\in S$ or $a,b\in S+r$, then we obviously have $\ord_2(a-b)<\ord_2r$.
   Therefore by symmetry we may assume that $a\in S$ and $b\in S+r$.
   If $a=b-r$, then we clearly have $\ord_2(a-b)=\ord_2r$;
   otherwise $a$ and $b-r$ are distinct elements of $S$ and so $\ord_2(a-(b-r))<\ord_2r$, which implies that $\ord_2(a-b)<\ord_2r$.
  \item
   Let $a,b\in S\cup(S+r)$ be distinct.
   If either $a,b\in S$ or $a,b\in S+r$, then we obviously have $\ord_2(a-b)>\ord_2r$.
   Therefore by symmetry we may assume that $a\in S$ and $b\in S+r$.
   Then we have $\ord_2(a-(b-r))>\ord_2r$, which implies that $\ord_2(a-b)=\ord_2r$.\qedhere
 \end{enumerate}
\end{proof}

In writing the proofs that follow, it is convenient to make the set-theoretical identification of a map $\varphi\colon A\to B$ with the its graph $\{(a,b)\in A\times B\mid\varphi(a)=b\}$.
Thus for maps $\varphi_1\colon A_1\to B_1$ and $\varphi_2\colon A_2\to B_2$,
the inclusion $\varphi_1\subset\varphi_2$ means that $A_1$ is a subset of $A_2$ and that the restriction of $\varphi_2$ to $A_1$ is $\varphi_1$;
the union $\varphi_1\cup\varphi_2$ represents a map from $A_1\cup A_2$ to $B_1\cup B_2$,
provided that $\varphi_1(a)=\varphi_2(a)$ for all $a\in A_1\cap A_2$.

\begin{lem}\label{lem:add_odd}
 Let $(X,\preceq)$ be a totally ordered set without isolated points.
 Let $S\subset\Q$ be finite and $r\in\Q^{\times}$ be such that $\ord_2(a-b)<\ord_2r$ for all distinct $a,b\in S$.
 If $f\colon S\to X$ is binary, $\tilde{S}=S\cup(S+r)$, and $x\in X\setminus\Ima f$,
 then there exists a binary map $\tilde{f}\colon\tilde{S}\to X$ with $f\subset\tilde{f}$ and $x\in\Ima\tilde{f}$.
\end{lem}

\begin{proof}
 Take the enumeration $a_1,\dots,a_m$ of $S$ so that $f(a_1)\prec\dots\prec f(a_m)$.
 We define an injection $\tilde{f}\colon\tilde{S}\to X$ with $f\subset\tilde{f}$ and $x\in\Ima\tilde{f}$
 by inductively defining $\tilde{f}(a_i+r)$ as an element of $X$
 such that $f(a_{i-1})\prec\tilde{f}(a_i+r)$, $\tilde{f}(a_{i-1}+r)\prec\tilde{f}(a_i+r)$, $\tilde{f}(a_i+r)\prec f(a_{i+1})$, and
 $\tilde{f}(a_i+r)\ne f(a_i)$ for $i=1,\dots,m$
 (ignore any condition in which $0$ or $m+1$ appears as a subscript),
 using the assumption that $X$ has no isolated points
 and choosing $x$ as $\tilde{f}(a_i+r)$ at the earliest opportunity.

 It remains to prove that $\tilde{f}$ is binary.
 For ease of notation, we define an equivalence relation $\sim$ on $\tilde{S}$
 by setting $a\sim b$ if and only if $a-b\in\{0,\pm r\}$;
 Lemma~\ref{lem:S_and_S+r} (2) implies that this is indeed an equivalence relation because
 $\ord_2(\pm2r)=\ord_2r+1>\ord_2r$.
 Note that for $a,a',b,b'\in\tilde{S}$ with $a\sim a'\not\sim b\sim b'$, we have $\tilde{f}(a)\prec\tilde{f}(b)$ if and only if $\tilde{f}(a')\prec\tilde{f}(b')$.
 Suppose that $\tilde{f}$ is not binary.
 Then we may find $a,b,c\in\tilde{S}$ such that $\tilde{f}(a)\prec\tilde{f}(b)\prec\tilde{f}(c)$ and $\ord_2(b-a)=\ord_2(c-b)=n$, say.
 Lemma~\ref{lem:S_and_S+r} (2) shows that $n\le\ord_2r$ and moreover that $n<\ord_2r$
 because $a\sim b\sim c$ would contradict the fact that $a$, $b$, $c$ are distinct elements of $\tilde{S}$.
 Therefore if we choose $a',b',c'\in S$ so that $a\sim a'$, $b\sim b'$, $c\sim c'$, then
 $a\sim a'\not\sim b\sim b'\not\sim c\sim c'$ and so $\tilde{f}(a')\prec\tilde{f}(b')\prec\tilde{f}(c')$ and $\ord_2(b'-a')=\ord_2(b-a)=\ord_2(c-b)=\ord_2(c'-b')$;
 but this is a contradiction because the restriction of $\tilde{f}$ to $S$ is equal to the binary map $f$.
 Hence we have proved that $\tilde{f}$ is binary.
\end{proof}

\begin{prop}\label{prop:construction_N,Z}
 If $S\in\{\N,\Z\}$, and $(X,\preceq)$ is a countably infinite totally ordered set without isolated points,
 then there exists a binary bijection $f\colon S\to X$.
\end{prop}

\begin{proof}
 Define a sequence $r_0,r_1,\dots\in S$ by
 \[
  r_n=
  \begin{cases}
   2^n&\text{if $S=\N$};\\
   (-2)^n&\text{if $S=\Z$},
  \end{cases}
 \]
 and set $S_n=\{\sum_{i\in A}r_i\mid A\subset\{0,\dots,n-1\}\}$ for $n\in\N$.
 Note that for $n\in\N$, since
 \[
  S_n=
  \begin{cases}
   [0,2^n-1]\cap\N&\text{if $S=\N$};\\
   [-2(2^n-1)/3,(2^n-1)/3]\cap\Z&\text{if $S=\Z$ and $n$ is even};\\
   [-2(2^{n-1}-1)/3,(2^{n+1}-1)/3]\cap\Z&\text{if $S=\Z$ and $n$ is odd},
  \end{cases}
 \]
 we have $\ord_2(a-b)<n=\ord_2r_n$ for all distinct $a,b\in S_n$.
 Fixing a bijection $g\colon\N\to X$, we shall inductively construct binary maps $f_n\colon S_n\to X$ for $n\in\N$ so that $f_0\subset f_1\subset\dotsb$.

 Define a binary map $f_0$ from $S_0=\{0\}$ to $X$ by setting $f_0(0)=g(0)$.
 Let $n\in\N$, and suppose that $f_n\colon S_n\to X$ has been constructed.
 Take the smallest $k\in\N$ with $g(k)\notin\Ima f_n$, and use Lemma~\ref{lem:add_odd} to
 construct a binary map $f_{n+1}\colon S_{n+1}\to X$ with $f_n\subset f_{n+1}$ and $g(k)\in\Ima f_{n+1}$.

 Define $f=\bigcup_{n=0}^{\infty}f_n$.
 Then $f$ is a binary map from $\bigcup_{n=0}^{\infty}S_n=S$ to $X$.
 Since the construction ensures that $g(0),\dots,g(k)\in\Ima f_k$ for all $k\in\N$,
 it follows that $f$ is surjective.
\end{proof}

\subsection{Construction of binary maps from $\Q$}
We write $\Z_{(2)}=\{r\in\Q\mid\ord_2r\ge0\}$.

\begin{lem}\label{lem:Z_2_binary_representation}
 Let $n\in\N$ and $q_0,\dots,q_{n-1}\in\Z_{(2)}$ with $\ord_2q_i=i$ for all $i=0,\dots,n-1$.
 Then for every $q\in\Z_{(2)}$ there exists a unique subset $A$ of $\{0,\dots,n-1\}$ such that $\ord_2(q-\sum_{i\in A}q_i)\ge n$.
\end{lem}

\begin{proof}
 We proceed by induction on $n$.
 The lemma is trivial for $n=0$ because the only choice for $A$ is $\emptyset$.
 Suppose that the lemma is true for a nonnegative integer $n$
 and that $q_0,\dots,q_n\in\Z_{(2)}$ satisfy $\ord_2q_i=i$ for all $i=0,\dots,n$.

 Let $q\in\Z_{(2)}$ be given.
 The inductive hypothesis shows that there exists a unique subset $A'$ of $\{0,\dots,n-1\}$ such that $\ord_2(q-\sum_{i\in A'}q_i)\ge n$.

 We first show the existence of the required subset $A$ of $\{0,\dots,n\}$.
 If $\ord_2(q-\sum_{i\in A'}q_i)\ge n+1$, then take $A=A'$;
 if $\ord_2(q-\sum_{i\in A'}q_i)=n$, then setting $A=A'\cup\{n\}$ gives $\ord_2(q-\sum_{i\in A}q_i)\ge n+1$
 because $\ord_2(\alpha+\beta)\ge n+1$ for all $\alpha,\beta\in\Q$ with $\ord_2\alpha=\ord_2\beta=n$.

 What remains to prove is the uniqueness of such $A$.
 Suppose that $A\subset\{0,\dots,n\}$ satisfies $\ord_2(q-\sum_{i\in A}q_i)\ge n+1$.
 Since $q-\sum_{i\in A\setminus\{n\}}q_i$ equals either $q-\sum_{i\in A}q_i$ or $q-\sum_{i\in A}q_i+q_n$,
 we must have $\ord_2(q-\sum_{i\in A\setminus\{n\}}q_i)\ge n$ because $\ord_2q_n=n$.
 Therefore it follows from the inductive hypothesis that $A\setminus\{n\}=A'$, which means that $A$ is either $A'$ or $A'\cup\{n\}$.
 If $\ord_2(q-\sum_{i\in A'}q_i)\ge n+1$ and $\ord_2(q-\sum_{i\in A'\cup\{n\}}q_i)\ge n+1$ were both true, then we would also have $\ord_2q_n\ge n+1$,
 contradicting the assumption that $\ord_2q_n=n$.
 This completes the proof of the uniqueness of $A$.
\end{proof}

\begin{lem}\label{lem:existence_q_n}
 There exists a sequence $q_0,q_1,\dotsc\in\Z_{(2)}$ with the following properties:
 \begin{enumerate}
  \item $\ord_2q_n=n$ for all $n\in\N$;
  \item for every $q\in\Z_{(2)}$ there exists a finite subset $A$ of $\N$ such that $q=\sum_{n\in A}q_n$.
 \end{enumerate}
\end{lem}

\begin{proof}
 Fix a bijection $h\colon\N\to\Z_{(2)}$.
 For each $n\in\N$,
 when $q_0,\dots,q_{n-1}\in\Z_{(2)}$ have been constructed,
 take the smallest $l_n\in\N$ such that $\ord_2(h(l_n)-\sum_{i\in A}q_i)=n$ for some $A\subset\{0,\dots,n-1\}$
 (such $l_n$ always exists because the condition follows for example from $\ord_2h(l_n)=n$;
 Lemma~\ref{lem:Z_2_binary_representation} implies that $A$ is unique),
 and set $q_n=h(l_n)-\sum_{i\in A}q_i$.

 Suppose that (2) is false, and
 take the smallest $l\in\N$ for which there is no finite subset $A$ of $\N$ such that $h(l)=\sum_{i\in A}q_i$.
 Choose $N\in\N$ so large that for every $k=0,\dots,l-1$ there exists $A_k\subset\{0,\dots,N-1\}$ such that $h(k)=\sum_{i\in A_k}q_i$.
 By Lemma~\ref{lem:Z_2_binary_representation}, there exists $A_l\subset\{0,\dots,N-1\}$ such that $\ord_2(h(l)-\sum_{i\in A_l}q_i)\ge N$.
 Set $N'=\ord_2(h(l)-\sum_{i\in A_l}q_i)$.
 Since $\ord_2(h(k)-\sum_{i\in A_k}q_i)=\infty>N'$ for $k=0,\dots,l-1$,
 we must have $l_{N'}\ge l$.
 Since $l$ has the property required for $l_{N'}$, this means that $l_{N'}=l$,
 contradicting the choice of $l$.
\end{proof}

\begin{lem}\label{lem:existence_r_n}
 There exists a sequence $r_0,r_1,\dotsc\in\Q^{\times}$ such that if we set $S_n=\{\sum_{i\in A}r_i\mid A\subset\{0,\dots,n-1\}\}$ for $n\in\N$, we have the following:
 \begin{enumerate}
  \item $\bigcup_{n=0}^{\infty}S_n=\Q$;
  \item $\ord_2(a-b)<\ord_2r_n$ whenever $n\in\N$ is even and $a,b\in S_n$ are distinct;
  \item $\ord_2(a-b)>\ord_2r_n$ whenever $n\in\N$ is odd and $a,b\in S_n$.
 \end{enumerate}
\end{lem}

\begin{proof}
 Take a sequence $q_0,q_1,\dotsc\in\Z_{(2)}$ as in Lemma~\ref{lem:existence_q_n}, and define $r_0,r_1,\dotsc\in\Q^{\times}$ by
 \[
  r_n=
  \begin{cases}
   q_{n/2}&\text{if $n$ is even};\\
   2^{-(n+1)/2}&\text{if $n$ is odd}.
  \end{cases}
 \]

 To show (1), we observe that every $r\in\Q$ can be written (in fact uniquely) as $q+\sum_{j\in B}2^{-j}$
 with $q\in\Z_{(2)}$ and a finite set $B\subset\N_+$.

 To show (2), suppose that $n=2m$ for some $m\in\N_+$ and that $a,b\in S_{2m}$ are distinct
 (we may assume that $n\ge2$ because $S_0$ consists of only one element).
 Since $a-b$ is a nonzero rational number that can be written
 as a linear combination of $q_0,\dots,q_{m-1}$ and $2^{-1},\dots,2^{-m}$ with coefficients in $\{0,\pm1\}$,
 we have
 \begin{align*}
  \ord_2(a-b)&\le\max\{\ord_2q_0,\dots,\ord_2q_{m-1},\ord_22^{-1},\dots,\ord_22^{-m}\}\\
  &=m-1<m=\ord_2r_n.
 \end{align*}

 To show (3), suppose that $n=2m+1$ for some $m\in\N$ and that $a,b\in S_{2m+1}$.
 Since $a-b$ can be written as a linear combination of $q_0,\dots,q_m$ and $2^{-1},\dots,2^{-m}$ with coefficients in $\{0,\pm1\}$, we have
 \begin{align*}
  \ord_2(a-b)&\ge\min\{\ord_2q_0,\dots,\ord_2q_m,\ord_22^{-1},\dots,\ord_22^{-m}\}\\
  &=-m>-m-1=\ord_2r_n.\qedhere
 \end{align*}
\end{proof}

\begin{lem}\label{lem:add_outside}
 Let $(X,\preceq)$ be a totally ordered set without a maximum.
 Let $S\subset\Q$ be finite and $r\in\Q^{\times}$ be such that $\ord_2(a-b)>\ord_2r$ for all $a,b\in S$.
 If $f\colon S\to X$ is binary and $\tilde{S}=S\cup(S+r)$, then there exists a binary map $\tilde{f}\colon\tilde{S}\to X$ with $f\subset\tilde{f}$.
\end{lem}

\begin{proof}
 Take the enumeration $a_1,\dots,a_m$ of $S$ so that $f(a_1)\prec\dots\prec f(a_m)$.
 We define an injection $\tilde{f}\colon\tilde{S}\to X$ with $f\subset\tilde{f}$
 by inductively defining $\tilde{f}(a_i+r)$ as an element of $X$ larger than $f(a_1),\dots,f(a_m),\tilde{f}(a_1+r),\dots,\tilde{f}(a_{i-1}+r)$,
 using the assumption that $X$ has no maximum.

 To prove that $\tilde{f}$ is binary, take any $a,b,c\in\tilde{S}$ with $\tilde{f}(a)\prec\tilde{f}(b)\prec\tilde{f}(c)$.
 We have either (i)~$a,b,c\in S$, (ii)~$a,b\in S$ and $c\in S+r$, (iii)~$a\in S$ and $b,c\in S+r$, or (iv)~$a,b,c\in S+r$.
 In cases (i) and (iv), we have $\ord_2(b-a)\ne\ord_2(c-b)$ because $f$ is binary.
 Lemma~\ref{lem:S_and_S+r} implies that $\ord_2(b-a)>\ord_2r=\ord_2(c-b)$ in case (ii) and that $\ord_2(b-a)=\ord_2r<\ord_2(c-b)$ in case (iii).
\end{proof}

\begin{prop}\label{prop:construction_Q}
 Let $(X,\preceq)$ be a countably infinite totally ordered set.
 Suppose that $X$ has no isolated points and
 that either $X$ does not have a maximum or $X$ does not have a minimum.
 Then there exists a binary bijection $f\colon\Q\to X$.
\end{prop}

\begin{proof}
 By symmetry, we may assume that $X$ does not have a maximum.
 Take a sequence $r_0,r_1,\dotsc\in\Q^{\times}$ as in Lemma~\ref{lem:existence_r_n},
 and set $S_n=\{\sum_{i\in A}r_i\mid A\subset\{0,\dots,n-1\}\}$ for $n\in\N$.
 Fixing a bijection $g\colon\N\to X$,
 we shall inductively construct binary maps $f_n\colon S_n\to X$ for $n\in\N$ so that $f_0\subset f_1\subset\dotsb$.

 Define a binary map $f_0$ from $S_0=\{0\}$ to $X$ by setting $f_0(0)=g(0)$.
 Let $n\in\N$, and suppose that $f_n\colon S_n\to X$ has been constructed.
 If $n$ is even, then take the smallest $k\in\N$ with $g(k)\notin\Ima f_n$, and use Lemma~\ref{lem:add_odd} to
 construct a binary map $f_{n+1}\colon S_{n+1}\to X$ with $f_n\subset f_{n+1}$ and $g(k)\in\Ima f_{n+1}$.
 If $n$ is odd, then use Lemma~\ref{lem:add_outside} to construct a binary map $f_{n+1}\colon S_{n+1}\to X$ with $f_n\subset f_{n+1}$.

 Define $f=\bigcup_{n=0}^{\infty}f_n$.
 Then $f$ is a binary map from $\bigcup_{n=0}^{\infty}S_n=\Q$ to $X$.
 Since the construction ensures that $g(0),\dots,g(k)\in\Ima f_{2k+1}$ for all $k\in\N$,
 it follows that $f$ is surjective.
\end{proof}

Propositions~\ref{prop:construction_N,Z} and \ref{prop:construction_Q} complete the proof of the `if' parts of our main theorem.

\section*{Acknowledgements}
This work was supported by JSPS KAKENHI Grant Numbers JP18K03243, JP18K13392, JP22K03244, and JP23K03072.
The authors would like to thank Hideki Murahara for valuable discussions.

\end{document}